\documentclass[11pt]{amsart} 
\usepackage[english]{babel}
\usepackage{graphicx,epsfig}
\usepackage{bbm}
\usepackage{cite}
\usepackage{listings}
\usepackage{amsthm}
\usepackage{enumerate}
\usepackage{esvect}
\usepackage{amsmath,amssymb,latexsym, amsfonts, amscd, amsthm, xy}
\input{xy}
\xyoption{all}

\usepackage{mathrsfs}
\usepackage{appendix}

\usepackage{mathtools}

\usepackage[margin=1in]{geometry}

\usepackage[dvipsnames]{xcolor}

\makeindex 

\usepackage{hyperref}
\hypersetup{pdftoolbar=true, pdftitle={GKform}, pdffitwindow=true, colorlinks=true, citecolor=blue, filecolor=black, linkcolor=purple, urlcolor=red, hypertexnames=false}

\theoremstyle{plain}

\newtheorem{theorem}{Theorem}[section]

\newtheorem{proposition}[theorem]{Proposition}

\newtheorem*{assumption*}{Assumption}
\newtheorem{lemma}[theorem]{Lemma}

\theoremstyle{definition}

\newtheorem{remark}{Remark}[section]

\newtheorem*{goal*}{Goal}
\newtheorem*{problem*}{Comment}

\def\lra{{\longrightarrow}}

\DeclareMathOperator{\Gr}{Gr}

\DeclareMathOperator{\Fil}{Fil}

\DeclareMathOperator{\dR}{dR} 
\DeclareMathOperator{\id}{id}

\DeclareMathOperator{\et}{et} \DeclareMathOperator{\Norm}{Norm}

 \DeclareMathOperator{\CH}{CH}
\DeclareMathOperator{\AJ}{AJ} \DeclareMathOperator{\Aut}{Aut}

\DeclareMathOperator{\ord}{ord}

\DeclareMathOperator{\rank}{rank}

\DeclareMathOperator{\alg}{alg}

\DeclareMathOperator{\im}{Im}

\DeclareMathOperator{\Jac}{Jac}

\DeclareMathOperator{\old}{old}

\def\p{\mathfrak{p}}

\def\d{\mathrm{d}}

\def\Z{\mathbb{Z}}

\def\F{\mathbb{F}}

\def\Q{\mathbb{Q}}

\def\C{\mathbb{C}}

\def\bdf{\begin{defn}}
\def\edf{\end{defn}}

\def\Gal{{\rm Gal}}

\def\d1{d^{(1)}}

\def\d{\mathbf{d}}

\usepackage{tikz}
\usetikzlibrary{positioning} 
\usepackage{tikz-cd}
\usetikzlibrary{arrows,graphs,decorations.pathmorphing,decorations.markings}

\tikzset{
commutative diagrams/.cd,
arrow style=tikz,
diagrams={>=latex}}

\makeatletter
\let\@wraptoccontribs\wraptoccontribs
\makeatother

\begin{document} 

 \title[Experiments with Ceresa classes of cyclic Fermat quotients]{Experiments with Ceresa classes of cyclic Fermat quotients}
\author{David T.-B. G. Lilienfeldt and Ari Shnidman}
 \address{Einstein Institute of Mathematics, Hebrew University of Jerusalem, Israel}
\email{davidterborchgram.lilienfeldt@mail.huji.ac.il}
\address{Einstein Institute of Mathematics, Hebrew University of Jerusalem, Israel}
\email{ariel.shnidman@mail.huji.ac.il}
\date{May 4, 2022}
\subjclass{11D41, 11G15, 11G40, 14C25}
\keywords{Abel-Jacobi map, Ceresa cycle, Fermat quotient, Beilinson-Bloch, Complex multiplication}

\begin{abstract}
We give two new examples of non-hyperelliptic curves whose Ceresa cycles have torsion images in the intermediate Jacobian.  For one of them, the central value of the $L$-function of the relevant motive is non-vanishing and the Ceresa cycle is torsion in the Griffiths group, consistent with the conjectures of Beilinson and Bloch. We speculate on a possible explanation for the existence of these torsion Ceresa classes, based on some computations with cyclic Fermat quotients. 
\end{abstract}

\maketitle

\section{Introduction}

Let $C$ be a smooth projective curve over a field $k$ with a point $e \in C(k)$. Let $J$ be the Jacobian of $C$, and let $\iota \colon C \hookrightarrow J$ be the Abel--Jacobi map sending $e$ to $0$.  The {\it Ceresa cycle} $\kappa_e(C)$ is the class of $\iota(C) - [-1]^*\iota(C)$ in the Chow group $\CH_1(J)_0$ of null-homologous $1$-cycles modulo rational equivalence on $J$. Let $\kappa(C)$ be the image of $\kappa_e(C)$  in the Griffiths group $\Gr_1(J)$ of null-homologous $1$-cycles modulo algebraic equivalence, which does not depend on the choice of  $e$. When $C$ is hyperelliptic, we have $\kappa(C) = 0$. On the other hand, Ceresa showed that $\kappa(C)$ has infinite order for the general curve of genus $g \geq 3$ \cite{ceresa}. It was speculated that this may be the case for all non-hyperelliptic curves, but recently two counterexamples were discovered.

Bisogno, Li, Litt, and Srinivasan  showed that for the genus 7 Fricke--Macbeath curve $C_{\mathrm{FM}}$, the image of $\kappa_e(C_{\mathrm{FM}})$ under the \'etale Abel--Jacobi map is torsion \cite{blls}. Gross then showed that the appropriate $L$-function has non-vanishing central value, giving more evidence that $\kappa(C_{\mathrm{FM}})$ is torsion in the Griffiths group \cite{grossFM}, as we will explain. Qiu and Zhang have recently shown that the corresponding Gross--Kudla--Schoen cycle vanishes in the Chow group \cite{qiuzhang}. In particular, this implies that $\kappa(C_{\mathrm{FM}})$ is torsion for the Fricke--Macbeath curve \cite[Rem.\ 3.4]{FLV}.

The second example is due to Beauville who shows in  \cite{beauville} that for the genus 3 plane curve $D^9 \colon y^3 = x^4 +x$, the complex Abel--Jacobi image of $\kappa_e(D^9)$ is torsion, for appropriate choice of $e$. Assuming the conjectural injectivity of the Abel--Jacobi map, this implies that $\kappa_e(D^9)$ and $\kappa(D^9)$ are torsion. Beauville and Schoen then gave an unconditional proof that $\kappa(D^9)$ is torsion in  \cite{beauvilleschoen}.

\subsection{New examples}
 
Beauville's proof in \cite{beauville} is elegant and simple, making use of an automorphism of order 9 on the curve. The aim of this note is to find more non-hyperelliptic examples of Ceresa cycles which are Abel--Jacobi trivial via Beauville's method. The curves we consider are the cyclic quotients of the Fermat curves $F^m : X^m+Y^m+Z^m=0$.  Beauville's curve $D^9$ is a cyclic quotient of $F^9$ (see Proposition \ref{app:henn}), so this is a natural family to consider. Ceresa cycles of cyclic Fermat quotients have been studied quite a bit \cite{eskandarimurty,kimura,otsubo, tadokoro2}, but mostly when $m$ is prime or when the Jacobian has many simple factors, which perhaps explains why these examples have been missed.   

We have turned Beauville's method into an algorithm and implemented it in SageMath \cite{sagemath}. 
On the positive side, we found two new examples:
\begin{theorem}
Let $D^{12}$ and $D^{15}$ be the smooth projective curves with affine equations $y^3 = x^4 + 1$ and $y^3 = x^5 + 1$, of genus $3$ and $4$ respectively, and let $e$ be the point at infinity. Then the images of $\kappa_e(D^{12})$ and $\kappa_e(D^{15})$ under the complex Abel--Jacobi map are torsion. 
\end{theorem}

On the negative side, we ran the algorithm on all cyclic Fermat quotients of $F^m$ up to $m=50$, and these were the only examples we found. Enlarging $m$ seems unlikely to discover more examples by Beauville's method (see Section \ref{s:discussion}). 

However, our experiments uncovered a phenomenon that points to a possible explanation of the three torsion Ceresa examples $D^9,D^{12}$, and $D^{15}$. In all three cases, there is an algebraic correspondence between $D^m$ and a hyperelliptic quotient of $F^m$.  
We use work of Aoki to show that the minimal cyclic Fermat quotients with this property and with $m \not\equiv 2\pmod{4}$ are our curves $D^9, D^{12},D^{15}$, and a genus 9 quotient of $F^{21}$.  This suggests the possibility that the Ceresa cycle is torsion because some multiple of it arises via correspondence from a hyperelliptic Ceresa class.  

The kernel of the Abel--Jacobi map (restricted to cycles defined over $\bar \Q$) is conjectured to be torsion. This would imply that the Ceresa cycle of each of the curves $D^9$, $D^{12}$, and $D^{15}$ (with respect to a certain base point) is torsion modulo rational equivalence, and hence modulo algebraic equivalence as well. For $D^9$, the latter statement is the main result of \cite{beauvilleschoen}. The argument of \cite{beauvilleschoen} applies to the curve $D^{12}$ as well, resulting in the following:

\begin{theorem}\label{D12tors}
The class $\kappa(D^{12})$ is torsion in the Griffiths group.
\end{theorem}

\subsection{$L$-function calculations}

The existence of non-torsion algebraic cycles in Chow groups is the main theme of the conjecture of Beilinson and Bloch \cite{bloch}. Given a smooth projective variety $X$ over a number field $K$, the conjecture predicts that the rank of the group $\CH^r(X)_0$ of null-homologous cycles of codimension $r$ is equal to the order of vanishing of the $L$-function of the motive $H^{2r-1}(X)$ in degree $2r-1$ at $s=r$. Similarly, refinements of this conjecture predict that in certain situations, $L$-functions of submotives of $H^{2r-1}(X)$ detect properties of cycles modulo algebraic equivalence (see Section \ref{sec:BB}). 

Both curves $D^9$ and $D^{12}$ have genus $3$, while $D^{15}$ has genus $4$. Let $J^9, J^{12},$ and $J^{15}$ denote their respective Jacobians. Since Jacobians of Fermat curves (and hence Jacobians of their quotients) admit complex multiplication, there are submotives $M^9$ and $M^{12}$ of type $(3,0)+(0,3)$ inside $H^3(J^9)$ and $H^3(J^{12})$ respectively. They are defined over the totally real fields $\Q(\zeta_9)^+$ and $\Q(\zeta_{12})^+$ respectively, and they are predicted by Bloch to account for the algebraic equivalence classes of the Ceresa cycles of $D^9$ and $D^{12}$. Similarly, there is a submotive $M^{15}$ of type $(4,1)+(1,4)$ inside $H^5(J^{15})$. It is defined over $\Q(\zeta_{15})^+$, and is predicted to account for $\kappa(D^{15})$.
Using Magma \cite{Magma}, we computed the orders of vanishing of the relevant $L$-functions:

\begin{theorem}
We have 
\[
\begin{array}{lll}
    \ord_{s=2} L(M^9/\Q(\zeta_9)^+,s) & = & 1 \\
    \ord_{s=2} L(M^{12}/\Q(\zeta_{12})^+,s) &= & 0 \\
    \ord_{s=3} L(M^{15}/\Q(\zeta_{15})^+,s) & = & 2.
\end{array}
\]
\end{theorem}

The non-vanishing of the $L$-function of $M^{12}$ at its center, together with Theorem \ref{D12tors}, constitutes new evidence for the Beilinson--Bloch conjectures.
If we assume the Beilinson--Bloch conjectures, then $D^{12}$ gives an example of a non-hyperelliptic Jacobian over $\Q$ whose Griffiths group $\Gr_1(J)$ has rank 0;  see Section \ref{subsec:Lfunctions}.
Note that for Beauville's original example $D^9$, we found that the order of vanishing is $1$
despite the fact that $\kappa(D^9)$ is torsion. 

\subsection{Outline}

Section \ref{s:ceresa} provides background on Beauville's method and cyclic Fermat quotients. In Section \ref{s:algoresult}, we apply Beauville's method to a genus $4$ curve as a way of illustrating the algorithmic aspect of the method, and then present the results of our implementation in SageMath. In Section \ref{s:discussion}, we compare our results with some recent work of Eskandari and Murty. After some background on the Beilinson--Bloch conjecture and Jacobi sums, we compute the orders of vanishing of the relevant $L$-functions in Section \ref{s:Lfct}.  Finally, in Section \ref{s:speculation} we speculate on some connections with hyperelliptic cyclic Fermat quotients.

\section{Preliminaries}\label{s:ceresa}
Let $C$ be a smooth, projective, geometrically integral curve over a field $k$ of genus $g\geq 3$. We sometimes consider affine and possibly singular models of $C$, but the Ceresa cycles will always be for the smooth projective model.  Let $\kappa_e = \kappa_e(C)$ and $\kappa = \kappa(C)$ be the Ceresa cycles and classes defined in the introduction.  Let $J$ be the Jacobian of $C$.

\subsection{Beauville's method}\label{s:beauville}

Assume now that $k = \C$. In \cite{beauville} Beauville considers the image of $\kappa_e$ under the complex Abel--Jacobi map 
\begin{equation}\label{JacCAJ}
\AJ_C : \CH^{g-1}(J)_0\lra J^{g-1}(J) :=\frac{\Fil^{2} H^{3}_{\dR}(J)^\vee}{\im H_{3}(J(\C), \Z)}.
\end{equation}
This map is defined by the integration formula 
\[
\AJ_C(Z)(\alpha):=\int_{\partial^{-1}(Z)} \alpha, \qquad \text{for all } \alpha\in \Fil^{2} H^{3}_{\dR}(J),
\]
where $\partial^{-1}(Z)$ denotes any continuous $3$-chain in $J(\C)$ whose boundary is $Z$.

Beauville's idea to prove that $\AJ_C(\kappa_e)$ is torsion in certain instances is simple and elegant. Suppose that $C$ has an automorphism $\sigma$ that fixes a point $e\in C(k)$. The push-forward $\sigma_*$ then fixes the Ceresa cycle $\kappa_e$. By functoriality of Abel--Jacobi maps, $\AJ_C(\kappa_e)$ is then fixed by $(\sigma^*)^\vee$ in $J^{g-1}(J)$. If the complex torus $J^{g-1}(J)$ has only finitely many fixed points under $(\sigma^*)^\vee$, or equivalently by \cite[13.1.2]{langebirkenhake}, if $1$ is not an eigenvalue for the action of $\sigma$ on the tangent space $T^{g-1}(C)$ of $J^{g-1}(J)$ at $0$, then $\AJ_C(\kappa_e)$ must be torsion. 
The tangent space $T^{g-1}(C)$ is given by 
\begin{equation}\label{tangent}
H^{g-2, g-1}(J) \oplus H^{g-3, g}(J) = \left(\bigwedge\nolimits^{g-2} V\otimes \bigwedge\nolimits^{g-1} V^* \right) \oplus  \left(\bigwedge\nolimits^{g-3} V\otimes \bigwedge\nolimits^{g} V^*\right),  
\end{equation}
where $V:=H^{1,0}(J)=H^0(C, \Omega_C^1)$ and $V^*$ denotes its $\C$-dual. 

Beauville carries out this idea for the curve $y^3 = x^4 +x$  using the order $9$ automorphism sending $(x, y) \mapsto (\zeta_9^3x, \zeta_9 y)$. Here $\zeta_9$ denotes a fixed choice of primitive $9$-th root of unity. 

\begin{remark}
Roughly speaking, one can view Beauville's strategy as a Hodge-theoretic analogue of the $\ell$-adic approach of \cite{blls}, but the latter approach is rather more elaborate.  Beauville's approach is more convenient to implement because it is enough to consider a single automorphism. 
\end{remark}

\subsection{Cyclic Fermat quotients}\label{s:fermat}

We collect some results on cyclic Fermat quotients, from  \cite{colemanfermat} and the references therein. Recall that $m$ denotes a positive integer. For any two integers $a, b$ satisfying $0<a,b<m$ and $\gcd(m, a, b, a+b)=1$, let $C^m_{a, b}$ be the smooth projective curve birational to
\begin{equation}\label{cFq}
F^m_{a, b} : y^m=(-1)^{a+b}x^a(1-x)^b.
\end{equation}
Denote by $J^m_{a,b}$ the Jacobian of $C^m_{a, b}$.
The map $F^m \lra F^m_{a,b}$ sending $(x \colon y \colon 1) \mapsto (-x^m, x^ay^b)$ induces a map $f_{a,b}^m \colon F^m \lra C^m_{a,b}$. 
The genus of $C_{a, b}^m$ is 
\begin{equation}\label{genus}
\frac{1}{2}(m-(\gcd(m, a)+\gcd(m, b)+\gcd(m, a+b))+2).
\end{equation}

If two other integers $a'$ and $b'$ satisfy the relation
\begin{equation}\label{equiv}
\{ a, b, m-a-b \}\equiv \{ ta', tb', t(m-a'-b') \} \pmod m, \: \text{ for some } t\in (\Z/m\Z)^{\times},
\end{equation}
henceforth written $(a, b)\sim_m (a', b')$,
then there is an isomorphism $C^{m}_{a,b}\simeq C^m_{a', b'}$. We will consequently restrict our attention to equivalence classes $[a, b]$ of pairs with respect to $\sim_m$. 

\begin{proposition}\label{p:hyperelliptic}
The cyclic Fermat quotient $C^m_{a,b}$ is hyperelliptic if and only if
\[
(a, b)\sim_m (1,1) \qquad \text{ or } \qquad m=2n \text{ and } (a, b) \sim_m (1, n).
\]
\end{proposition}

\begin{proof}
This is \cite[\S IV Prop.\ 8 \& Cor.\ 8.1]{colemanfermat}.
\end{proof}

Denote by $\mu_m$ the group of $m$-th roots of unity and let $\zeta_m$ be a fixed choice of primitive $m$-th root of unity.
Note that $\mu_m$ acts on $C=C^m_{a,b}$ by scaling the $y$-coordinate in \eqref{cFq}. One can see the action of $\mu_m$ on the space $V=H^0(C, \Omega_C^1)$ by pulling back the differential forms in $V$ to the Fermat curve $F^m$ via the map $f=f_{a,b}^m$. 
A basis for $f^*V$ is 
\begin{equation}\label{diff}
L^m_{a,b}:=\{ w_{r,s} \: : \: 0< r, s, r+s < m \: \text{ and } \: br\equiv as \pmod m \},
\end{equation}
where $w_{r,s}:=w^m_{r,s}:=x^{r-m}y^{s-1}dy$ with functions $x=X/Z$ and $y=Y/Z$ on $F^m$.

For integers $1\leq i,j \leq m$, the automorphisms $\sigma_{i,j} : (x, y)\mapsto (\zeta_m^i x, \zeta_m^j y)$ of $F^m$ descend to the curve $C$. Note that some of them will be trivial, as $\Aut(F^m)$ contains $\mu_m^2$ and $\Aut(C)$ only contains $\mu_m$ in general. 
The automorphisms of $F^m$ act on the eigenbasis \eqref{diff} by the following rule:
\begin{equation}\label{eigen}
\sigma_{i,j}^*(w_{r,s})=\zeta_m^{ir+js}w_{r,s}.
\end{equation}
It follows that $\sigma_{i,j}$ acts on the dual space $V^*$ with eigenvalues 
\begin{equation}\label{eigend}
\{ \zeta_m^{-(ir+js)} \: : \: 0< r, s, r+s < m \: \text{ and } \: br\equiv as \pmod m \}.
\end{equation}
The eigenvalues of $\sigma_{i,j}$ acting on the tangent space \eqref{tangent} are computed from those on $V$ and $V^*$ using standard properties of wedge products.

\section{Algorithm}\label{s:algoresult}

\subsection{A genus $4$ example}\label{s:example}

Take $m=15$, $a=3$, $b=5$, and $\sigma=\sigma_{2,1}$.  Write $\zeta = \zeta_{15}$. By \eqref{genus}, the genus of $C=C^{15}_{3,5}$ is $4$. It follows from \eqref{tangent} that 
\[
T^{3}(C)=\left(\bigwedge\nolimits^{2} V\otimes \bigwedge\nolimits^{3} V^* \right) \oplus  \left(V\otimes \bigwedge\nolimits^{4} V^*\right).
\]
The eigenbasis of $V$ is 
$
L=L^{15}_{3,5}=\{ w_{3,5}, w_{3,10}, w_{6,5}, w_{9, 5} \}.
$
Using \eqref{eigen} and \eqref{eigend}, the eigenvalues of $\sigma$ acting on $V$ and $V^*$ are respectively 
\[
\{ \zeta^{11}, \zeta, \zeta^{2}, \zeta^{8} \} \quad \text{ and } \quad \{ \zeta^{4}, \zeta^{14}, \zeta^{13}, \zeta^{7} \}. 
\]
From properties of wedge products, we deduce that the eigenvalues of $\sigma$ acting on
\begin{itemize}
\item $\bigwedge^{2} V$ are $\{ \zeta^{12}, \zeta^{13}, \zeta^{4}, \zeta^{3}, \zeta^{9}, \zeta^{10} \}$, 
\item $\bigwedge\nolimits^{3} V^*$ are $\{ \zeta, \zeta^9, \zeta^{10}, \zeta^{4} \}$,
\item $\bigwedge^{4} V^*$ is $\zeta^{8}$.
 \end{itemize}
The eigenvalues of $\sigma$ acting on $\bigwedge^{2} V\otimes \bigwedge^{3} V^*$ and $V\otimes \bigwedge\nolimits^{4} V^*$ are therefore
\[
\{ \zeta^{13}, \zeta^{6}, \zeta^{7}, \zeta,
\zeta^{14}, \zeta^{8}, \zeta^{2},
\zeta^{5},
\zeta^{4}, \zeta^{12},
\zeta^{10}, \zeta^{3},
\zeta^{11}
 \}
 \quad \text{ and } \quad
 \{ \zeta^{4}, \zeta^{9}, \zeta^{10}, \zeta \},
\]
respectively. We conclude that $1$ is not an eigenvalue of $\sigma$ acting on $T^3(C)$. By the arguments of Section \ref{s:beauville}, $\AJ_C(\kappa_e)$ is torsion for any point $e\in C(\bar{\Q})$ fixed by $\sigma$.

\subsection{Algorithm and results}\label{s:algo}

We have implemented Beauville's method in SageMath \cite{sagemath}; see\cite{sagecode}. 
The algorithm takes as input a positive integer $m$ and outputs the list of all cyclic Fermat quotients of $F^m$ in the format $(m, a, b, m-a-b)$ together with their genus. 
If a cyclic Fermat quotient has an automorphism $\sigma_{i,j}$ for which $1$ is not an eigenvalue for the action on the tangent space \eqref{tangent}, then the algorithm also outputs $(m, a, b, i, j)$. Only the first such automorphism is listed, as the existence of an automorphism is what matters for our purpose.

We ran our algorithm for all positive integers $m$ up to $50$. For each even $m=2n$, the cyclic Fermat quotient $C^m_{1,n-1}$ has an automorphism for which $1$ is not an eigenvalue, hence the arguments of Section \ref{s:beauville} apply. However $(1, n-1) \sim_m (1, n)$ and thus by Proposition \ref{p:hyperelliptic} the curve $C^m_{1,n-1}$ is hyperelliptic. The only other examples of torsion Ceresa class we found were for the curves
\begin{equation}\label{output}
C^9_{1,2}, \: C^{12}_{1,3}, \: \text{ and } \: C^{15}_{1,5}. 
\end{equation}
The first two have genus $3$ and the third has genus $4$.
By Proposition \ref{p:hyperelliptic}, these are non-hyperelliptic.

\begin{proposition}\label{app:henn}
The curves $C^9_{1,2}, C^{12}_{1,3}$, and $C^{15}_{1,5}$ are respectively isomorphic over $\Q$ to the smooth projective plane curves with affine models
\[
\begin{array}{lll}
D^9 & : & y^3=x^4+x \\
D^{12} & : & y^3=x^4+1 \\
D^{15} & : & y^3=x^5+1. \\  
\end{array}
\]
\end{proposition}

\begin{proof}
Magma \cite{Magma} (see \cite{magmacode} for our code). 
\end{proof}

\begin{remark}
For each of the three curves $D^m$ above, it is known that $D^m$ is the unique curve over $\C$ of genus $g(D^m)$ with an automorphism of order $m$. For the genus 3 curves, see \cite[pg. 2]{lorenzogarcia}. For $D^{15}$, this follows from \cite[Table 4]{MSSV} which says that the locus of genus 4 curves admitting an automorphism subgroup of size $15$ is irreducible and 0-dimensional.
\end{remark}

\begin{theorem}\label{thm:AJtors}
Let $C$ be either $C^{9}_{1,2}, C^{12}_{1,3}$, or $C^{15}_{1,5}$, and let $e = (0,0)$. Then $\AJ_C(\kappa_e)$ is torsion.
\end{theorem}

\begin{proof}
In each case, $(0,0)$ is a non-singular point of $C^m_{1,b}$ fixed by the $\mu_m$-action. 
The result then follows from the arguments of Section \ref{s:beauville} and the output \eqref{output} of our algorithm (see \cite{sagecode} for our code).
\end{proof}

\begin{remark}
Note that \[(-2)\cdot \{3,5,7\} \equiv \{ 9, 5, 1 \} \equiv \{ 1, 5, 9 \} \pmod{15},\] which implies that $(3,5)\sim_{15} (1,5)$, and thus $C^{15}_{1,5}\simeq C^{15}_{3,5}$. Theorem \ref{thm:AJtors} therefore recovers the example calculated by hand in Section \ref{s:example}.
\end{remark}

\begin{remark}
Proposition \ref{app:henn} and Theorem \ref{thm:AJtors} generalize Beauville's result \cite{beauville} for the curve $D^9$.
\end{remark}

Now that we have found more curves with this property, we can ask whether the method of Beauville--Schoen in \cite{beauvilleschoen} applies to show that the Ceresa cycle is torsion in the Griffiths group. For the genus $3$ curve $C^{12}_{1,3}$, we have:

\begin{theorem}\label{thm:Griffiths}
The class $\kappa(C^{12}_{1,3})$ is torsion in the Griffiths group $\Gr^2(J^{12}_{1,3})$.
\end{theorem}

\begin{proof}
We work with the projective model $D:= D^{12} : Y^3Z=X^4+Z^4$ of the curve $C^{12}_{1,3}$. Let $J:=\Jac(D^{12})$ and $\zeta:=i\zeta_3$ be our choice of primitive $12$-th root of unity. Consider the order $12$ automorphism $\sigma : (X,Y,Z)\mapsto (iX, \zeta_3Y,Z)$, which admits a fixed point $p:=(0,1,0)$. Denote by $W$ the quotient variety $J/\langle \sigma \rangle$ and by $q : J\lra W$ the corresponding quotient map. If $F\subset J$ denotes the set of elements with non-trivial stabilizer under the action of $\langle \sigma\rangle$, then $q(F)$ is the singular locus $S$ of $W$ \cite[Cor. of Lem.\ 2]{Fujiki}. Let $J^0:=J\setminus F$ and $W^0:=W\setminus S$. The eigenvalues of $\sigma$ acting on $H^0(D,\Omega^1_{D})^*$ are $\zeta^{2}, \zeta^{7}, \zeta^{11}$. It follows that for $0<d\neq 6<12$, $\ker(1_J-\sigma^d)$ is finite. However, $\ker(1_J-\sigma^6)$ is $1$-dimensional. We now give a description of the $1$-dimensional part of $F$. 

Denote by $\iota$ the involution $\sigma^6 : (x,y)\mapsto (-x,y)$ of $D$. Consider the elliptic curve $E : y^3 = x^2+1$ and observe that there is a double covering $\pi : D \lra E$ given by $(x,y)\mapsto (x^2,y)$, which ramifies at  four branch points.
Pullback of divisors yields a map $\pi^* : E \lra J$, which is injective since $\pi$ is ramified \cite[\S3, Lem.]{mumfordprym}. We will therefore identify $E$ with its image $\pi^*(E)$ inside $J$. Let $\lambda_E$ and $\lambda_J$ denote the canonical principal polarizations of $E$ and $J$ respectively, and let $\hat{\pi} : \hat{E}\lra \hat{J}$ be the dual morphism of $\pi : J\lra E$. Then we have $\pi^*=\lambda_J^{-1}\circ \hat{\pi} \circ \lambda_E$. Since $\pi^*$ is injective, so is $\hat{\pi}$. The kernels of $\pi$ and $\hat{\pi}$ have the same number of connected components \cite[Prop. 4.6]{Debarre}. It follows that $A:=\ker(\pi)$ is connected, hence an abelian subvariety of $J$, namely the Prym variety of $\pi$. The involution $\iota$ of $D$ induces an involution $\iota : J\lra J$, which acts as $+1$ on $E\subset J$ and as $-1$ on $A\subset J$. In fact, we have $A=\ker(1_J+\iota)^0=\im(1_J-\iota)$. 
There is an isomorphism $J\simeq (E\times A)/(\id, \pi^*)(E[2])$ \cite[\S3, Cor.\ 1 of Lem.]{mumfordprym}. 
In conclusion, we have $E=\ker(1_J-\iota)^0$ and $\ker(1_J-\iota)$ is the union of the translates of $E$ by representatives of $A[2]/\pi^*(E[2])$. 

Observe that $\sigma^7:(x,y)\mapsto (-ix,\zeta_3y)$ acts on $H^0(D,\Omega^1_D)^*$ with eigenvalues $\zeta, \zeta^2, \zeta^5$. Since the singular locus of $W$ occurs in codimension $2$ and $1+2+5=8<12$, Reid's criterion \cite[Thm.\ 3.1]{Reid} implies that the singularities $q(x)$ for $x\in \ker(1_J-\sigma^7)$ are non-canonical. It follows exactly as in \cite[Lem.\ 2]{beauvilleschoen} that the threefold $W$ is uniruled. Following the proof of \cite[Thm.]{beauvilleschoen} verbatim, we deduce that $\kappa(D)$ is trivial in the group $A_1(J^0)\otimes \Q$ of $1$-cycles on $J^0=J\setminus F$ with coefficients in $\Q$, modulo algebraic equivalence. That is, the Ceresa class is torsion in the open subset of $J$ where $\langle \sigma \rangle$ acts freely. In the case treated in \cite{beauvilleschoen}, the set $F$ is 0-dimensional, hence the latter statement is equivalent to the Ceresa class being torsion in $J$. In the present case, the locus $F$ has a $1$-dimensional part given by a finite union of translates of $E$. By \cite[Ex. \ 10.3.4]{fulton}, a multiple of $\kappa(D)$ is algebraically equivalent to a linear combination of these translates of $E$. But these translates are algebraically equivalent to $E$, hence a multiple of $\kappa(D)$ is algebraically equivalent to some multiple of $E$. Since $\kappa(D)$ is null-homologous, the latter multiple must be zero. In conclusion, $\kappa(D)$ is torsion in the Griffiths group $\Gr_1(J)$. 
\end{proof}

\subsection{Discussion and related work}\label{s:discussion}

The genus of $C^m_{a,b}$ is of size $O(m)$ as $m \to \infty$. The dimension of the tangent space \eqref{tangent} therefore grows as $O(m^3)$. Since there are $O(m)$ choices for eigenvalues, it becomes very unlikely that $1$ is not an eigenvalue for the action of $\sigma_{i,j}$ as $m \to \infty$. It therefore seems unlikely that Beauville's method alone will provide more examples of non-hyperelliptic curves with torsion Ceresa cycle under the complex Abel--Jacobi map. The output of our algorithm for large $m$ is consistent with this heuristic.

Our results can  be compared with recent work of Eskandari and Murty \cite{eskandarimurty}. They show that for every choice of base point, the image of the Ceresa cycle of $F^m$ under the complex Abel--Jacobi map is non-torsion, for any integer $m$ divisible by a prime greater than $7$. 
For an excellent survey of results concerning the non-triviality of Ceresa cycles of Fermat curves and their quotients see \cite{eskandarimurty2}. In \cite[Rem.\ 2, pg. 17]{eskandarimurty2}, Eskandari and Murty speculate whether their method can be extended to prove that for every choice of base point, the image of the Ceresa cycle of the cyclic Fermat quotient $C^p_{1,b}$ under the complex Abel--Jacobi map is non-torsion for all prime numbers $p>7$. Indeed, we ran our algorithm for quotients of Fermat curves of prime degree up to $71$ and found that Beauville's method applies to none of them. It is interesting that for composite values of $m$, we do find curves $C^m_{1,b}$ with torsion Ceresa cycle, but only when the prime factors $p \mid m$ all satisfy $p < 7$.

\section{$L$-functions}\label{s:Lfct}

We make explicit the Beilinson--Bloch conjecture for the Jacobians $J^{9}_{1,2}, J^{12}_{1,3}$, and $J^{15}_{1,5}$, which predicts the rank of the Griffiths group of $1$-cycles in terms of certain $L$-functions.  For $J^{12}_{1,3}$, we give new evidence for this conjecture. 

\subsection{The Beilinson--Bloch conjecture}\label{sec:BB}

Let $C$ be as in Section \ref{s:ceresa}, defined over a number field $K$. The conjecture of Beilinson and Bloch \cite{bloch} predicts the equality
\begin{equation}\label{bb}
\rank_{\Z} \CH^{g-1}(J)_0 = \ord_{s=g-1} L(H^{2g-3}(J)/K, s),
\end{equation}
where $L(H^{2g-3}(J)/K, s)$ denotes the $L$-function of the compatible system of $\ell$-adic Galois representations coming from the cohomology of the Jacobian  $J$ in degree $2g-3$. Both the finite generation of $\CH^{g-1}(J)_0$ and the analytic continuation of the $L$-function are not known in general, hence are implicit in the conjecture.

Suppose there is a motivic decomposition $H^{2g-3}(J)=I\oplus M$, where $M$ is pure of type $(g, g-3)+(g-3,g)$ and $I$ has type $(g-1,g-2) + (g-2,g-1)$.
Then  the subgroup of algebraically trivial cycles $\CH^{g-1}(J)_{\alg}$ maps trivially under the \'etale Abel--Jacobi map to $H^1(\Gal(\bar K/K), M_{\mathrm{et}}(g-1))$ \cite{bloch}.
Hence, it is natural to expect the equality
\begin{equation}\label{bb2}
\rank_{\Z} \Gr^{g-1}(J) = \ord_{s=g-1} L(M/K, s),
\end{equation}
where $\Gr^{g-1}(J):=\CH^{g-1}(J)_0/\CH^{g-1}(J)_{\alg}$ is the Griffiths group. 
The equality \eqref{bb2} is Bloch's ``Son of Recurring Fantasy''. Even if $H^{2g-3}(J)$ does not have such a decomposition, Bloch defines a coniveau filtration $F^i \subset H^{2g-3}(J)$ and conjectures that the rank of $\Gr^{g-1}(J)$ is equal to the order of vanishing of $L(F^0/F^{g-2},s)$ at its central point. See \cite{bloch85} for more details.

Next, we recall an explicit description of the relevant $L$-functions in our examples.

\subsection{Jacobi sums}\label{s:jacobisum}
Fix $m, a$, and $b$ as in Section \ref{s:fermat}, and let $C = C^m_{a,b}$ and $J =J^{m}_{a,b}$. For any $n$ dividing $m$, there is a map $F^m \lra F^{m/n}$ given by $(x,y)\mapsto (x^n,y^n)$. Together with $f = f_{a,b}^m$, it induces a map of Jacobians $\Jac(F^{m/n}) \lra J$. Define $J^{\old}$ to be the subvariety of $J$ generated by the images of the above maps for all proper divisors $n$ of $m$. Define the {\it new part} $A = A^m_{a,b}:=J/J^{\old}$, a quotient of $J^m_{a,b}$.
Let
\[
H^{m}_{a,b}:=\{ h\in (\Z/m\Z)^\times \colon \langle ha \rangle+\langle hb \rangle < m \},
\]
where $\langle \cdot \rangle$ denotes the unique representative in $\{ 1, \ldots, m-1 \}$ of the residue class modulo $m$, and define
\[
W^m_{a,b}:=\{ h\in (\Z/m\Z)^\times \colon hH_{a,b}^m = H^m_{a,b} \}.
\]
The abelian variety $A$ has CM by $\Z[\zeta_m]$ with CM type equal to
$H^m_{a,b} \subset \Gal(\Q(\zeta_m)/\Q)$.
 Moreover, $A$ is simple if and only if $W^m_{a,b}=\{ 1 \}$. In general, $A$ is isogenous over $\C$ to $\vert W^m_{a,b}\vert$ isomorphic simple factors each having CM by the subfield $K^m_{a,b}$ of $\Q(\zeta_m)$ fixed by $W^m_{a,b}$ and CM type given by
 $H^m_{a,b}/W_{a,b}^m \subset \Gal(K^m_{a,b}/\Q)$ (see for instance \cite{koblitzrohrlich}). 

For each prime ideal $\p$ of $\Z[\zeta_m]$, let $\F_\p$ denote the residue field $\Z[\zeta_m]/\p$. If $\p\nmid m$, then the $m$-th power residue symbol $\chi_{\mathfrak{p}} : \F_{\mathfrak{p}}^\times \lra \mu_m \subset \Q(\zeta_m)$ is uniquely determined by the congruence  $\chi_{\mathfrak{p}}(z)=z^{\frac{N(\mathfrak{p})-1}{m}} \pmod \p$.
Define the Hecke character (defined on ideals of $\Z[\zeta_m]$ prime to $m$) 
\[\tau^m_{a,b}=\tau_{a,b} : I_{\Q(\zeta_m)}(m) \lra \Q(\zeta_m)^\times\] whose value on prime ideals $\p$ is the Jacobi sum
\[
\tau_{a,b}(\mathfrak{p}):=-\sum_{z\in \F_\mathfrak{p} \backslash \{0,1\}} \chi_{\mathfrak{p}}^a(z)\chi_{\mathfrak{p}}^b(1-z).
\]
The algebraic integer $\tau_{a,b}(\p)$ has absolute value $N(\mathfrak{p})^{1/2}$ in every complex embedding \cite{weiljacobi}. 

Fix a primitive character $\chi : \mu_m \lra \C^\times$ and define $\tau^{m,\chi}_{a,b}:=\tau^\chi_{a,b}:=\chi \circ \tau_{a,b} : I_{\Q(\zeta_m)}(m) \lra \C^\times$. This is a Grossencharacter for $\Q(\zeta_m)$ whose infinity type is given by the CM type of $A$. 
The complex $L$-function $L(A/\Q,s)$, attached to the cohomology of $A$ in degree $1$, has degree $\varphi(m)$ and we have the equality 
\[
L(A/\Q,s)=L(\tau^\chi_{a,b}/\Q(\zeta_m), s),
\]
where the latter is the Hecke $L$-function attached to the Grossencharacter $\tau^\chi_{a,b}$ \cite{weiljacobi}.  In particular, it has analytic continuation and satisfies a functional equation \cite{TateThesis}.

\begin{remark}
The equality of $L$-functions is independent of the choice of $\chi$.  Indeed,  the polynomial 
$
\prod_{\mathfrak{p}\mid p} (1-\tau_{a,b}(\mathfrak{p})T^{f_\mathfrak{p}})
$
has integer coefficients. 
\end{remark}

\subsection{$L$-function calculations}\label{subsec:Lfunctions}

Inspired by Bloch's $L$-function calculations \cite{bloch} for the Klein quartic $C^{7}_{2,1}$, we compute the order of vanishing of the $L$-function relevant for the algebraic equivalence class of the Ceresa cycle in the case of Beauville's curve $C^9_{1,2}$, as well as the curves $C^{12}_{1,3}$ and $C^{15}_{1,5}$. 
We refer to \cite{magmacode} for the Magma code used in this section.

\subsubsection{The $L$-function for $C^{9}_{1,2}$}

Fix a character $\chi : \mu_9 \lra \C^\times$ by sending $\zeta_9$ to $e^{\frac{2\pi i}{9}}$ and identify as usual $\Gal(\Q(\zeta_9)/\Q)$ with $(\Z/9\Z)^\times=\{ 1,2,4,5,7,8 \}$. We order the complex embeddings by pairs of complex conjugates as $\{ (1,8), (2,7), (4,5) \}$. In this subsection, $C = C^9_{1,2}$.

The morphism $f : F^9 \lra C$ is given by $(x \colon y \colon 1) \mapsto (-x^9, xy^2)$. We have $V:=H^{1,0}(C)$ and a basis of $f^*V$ is (in the notation of Section \ref{s:fermat})
\[
L^9_{1,2} = \{ w_{1,2}, w_{2,4}, w_{5,1} \}.
\]
The automorphism $\sigma(x\colon y\colon 1) = (\zeta_9^5 x \colon \zeta_9^7y \colon 1)$ of $F^9$ descends to an order $9$ automorphism of $C$. We have 
\begin{equation}\label{eigenC9}
\sigma^* w_{1,2}=\zeta_9 w_{1,2}, \qquad \sigma^* w_{2,4}=\zeta_9^2 w_{2,4}, \qquad \sigma^* w_{5,1}=\zeta_9^5 w_{5,1}. 
\end{equation}

Let $J = J^9_{1,2}$.  For each prime $\ell$, the $\ell$-adic cohomology group $H^1_{\et}(J_{\bar\Q}, \bar\Q_\ell)$ inherits an action of $\mu_9\subset \Aut(C)$ and admits an eigenspace decomposition
\[
H^1_{\et}(J_{\bar \Q}, \bar\Q_\ell)= \bigoplus_{\phi} H^{\phi},
\]
where the sum is over primitive characters $\phi : \mu_9 \lra \C^\times$ and $H^\phi$ is the corresponding eigenspace.
This is a direct sum of $\Gal(\bar\Q/\Q(\zeta_9))$-representations, and the decomposition carries an action of $\Gal(\Q(\zeta_9)/\Q)=(\Z/9\Z)^\times$ by $a(H^\phi)=H^{\phi^a}$. It follows from \eqref{eigenC9} that 
\begin{equation}\label{eq:decom10}
H^{1,0}(C)=H^{\chi}\oplus H^{\chi^2} \oplus H^{\chi^5},
\end{equation}
after base change to $\C$.

\begin{remark}
We have $J = J^9_{1,2}=A^9_{1,2}$, which is a simple abelian variety with CM by $\Z[\zeta_9]$. The decomposition \eqref{eq:decom10} is reflected by the fact that the CM type of $J$ is
 $H^9_{1,2}=\{ 1,2,5 \}$. 
\end{remark}

Let $H$ denote the motive $H^1(J)$.
The $(3,0)+(0,3)$ type submotive of $H^3(J)=\bigwedge^3 H$ is 
\begin{equation}\label{M9}
M = M^9_{1,2}:=\left( H^{\chi}\otimes H^{\chi^2} \otimes H^{\chi^5} \right)\oplus \left( H^{\chi^8}\otimes H^{\chi^7} \otimes H^{\chi^4} \right).
\end{equation}
The motive $M$ is defined over the maximal totally real subfield $\Q(\zeta_9)^+$ of $\Q(\zeta_9)$. The Beilinson--Bloch conjecture \eqref{bb2} in this case predicts that
\begin{equation}\label{bb9}
\rank_{\Z} \Gr^2(J_{\Q(\zeta_9)^+}) = \ord_{s=2} L(M/\Q(\zeta_9)^+, s).
\end{equation}

\begin{proposition}\label{prop:L9}
The order of vanishing of $L(M/\Q(\zeta_9)^+, s)$ at $s=2$ is equal to $1$.
\end{proposition}

\begin{proof}
From the discussion in Section \ref{s:jacobisum}, we have $L(H^{\chi}/\Q(\zeta_9),s) = L(\tau^\chi_{1,2},s)$, where $\tau^\chi_{1,2}$ has infinity type $[[1,0], [1,0], [0,1]]$. The conductor of this character has norm $3^4$ (see \cite[Table pg. 59]{cainthesis}).
We have the equalities of $L$-functions
\[
L(M/\Q(\zeta_9)^+, s)=L(\tau_{1,2}^{\chi}\tau_{1,2}^{\chi^2}\tau_{1,2}^{\chi^5}/\Q(\zeta_9), s)=L(\tau_{1,2}^{\chi}\tau_{2,4}^{\chi}\tau_{5,1}^{\chi}/\Q(\zeta_9), s), 
\]
where the latter is 
defined as 
\[
L(\tau_{1,2}^{\chi}\tau_{2,4}^{\chi}\tau_{5,1}^{\chi}/\Q(\zeta_9), s)=\prod_{\mathfrak{p} \: \triangleleft \: \Q(\zeta_9)} (1-\tau_{1,2}^{\chi}\tau_{2,4}^{\chi}\tau_{5,1}^{\chi}(\mathfrak{p})N(\mathfrak{p})^{-s})^{-1}=\prod_p L_{p}(s).
\]
If $\p$ is a prime of $\Q(\zeta_9)$ above $p \neq 3$, the factor $L_p(s)$ can be described as follows:

\vspace{2mm}
\noindent
$p\equiv 1\pmod 9:$ 
\begin{multline*}
(1-\tau_{1,2}^{\chi}\tau_{2,4}^{\chi}\tau_{5,1}^{\chi}(\mathfrak{p})p^{-s})^{-1}(1-\tau_{2,4}^{\chi}\tau_{4,8}^{\chi}\tau_{1,2}^{\chi}(\mathfrak{p})p^{-s})^{-1} 
(1-\tau_{4,8}^{\chi}\tau_{8,7}^{\chi}\tau_{2,4}^{\chi}(\mathfrak{p})p^{-s})^{-1} \\ (1-\tau_{5,1}^{\chi}\tau_{1,2}^{\chi}\tau_{7,5}^{\chi}(\mathfrak{p})p^{-s})^{-1}(1-\tau_{7,5}^{\chi}\tau_{5,1}^{\chi}\tau_{8,7}^{\chi}(\mathfrak{p})p^{-s})^{-1}(1-\tau_{8,7}^{\chi}\tau_{7,5}^{\chi}\tau_{4,8}^{\chi}(\mathfrak{p})p^{-s})^{-1}
\end{multline*}

\noindent
$p\equiv 2,5 \pmod 9: (1+p^9 p^{-6s})^{-1}$ 

\noindent
$p\equiv 8 \pmod 9: 
(1-\tau_{1,2}^{\chi}\tau_{2,4}^{\chi}\tau_{5,1}^{\chi}(\mathfrak{p})p^{-2s})^{-1}(1-\tau_{4,8}^{\chi}\tau_{8,7}^{\chi}\tau_{2,4}^{\chi}(\mathfrak{p})p^{-2s})^{-1}(1-\tau_{7,5}^{\chi}\tau_{5,1}^{\chi}\tau_{8,7}^{\chi}(\mathfrak{p})p^{-2s})^{-1}
$

\noindent
$p\equiv 4,7 \pmod 9: (1-\tau_{1,2}^{\chi}(\tau_{2,4}^{\chi})^2(\mathfrak{p})p^{-3s})^{-1}(1-(\tau_{1,2}^{\chi})^2\tau_{2,4}^{\chi}(\mathfrak{p})p^{-3s})^{-1}.$

\vspace{2mm}
The infinity type of
\[
\begin{cases}
\tau_{1,2}^\chi & \text{is } [[1,0], [1,0], [0,1]] \\
\tau_{2,4}^\chi & \text{is } [[1,0], [0,1], [0,1]] \\
\tau_{5,1}^\chi & \text{is } [[1,0], [1,0], [1,0]].
\end{cases}
\]
Thus, the infinity type of $\tau_{1,2}^\chi \tau_{2,4}^\chi \tau_{5,1}^\chi$ is $[[3,0],[2,1],[1,2]]$.

We can calculate the local factors of $L(M/\Q(\zeta_9)^+, s)$ in SageMath.  By comparing local $L$-factors at all primes above $p\in \{ 7,11,13,17,19,37 \}$, we can identify the Grossencharacter $\tau_{1,2}^\chi \tau_{2,4}^\chi \tau_{5,1}^\chi$ with one of the finitely many Hecke characters in Magma with appropriate infinity type and conductor:\footnote{Warning to the reader: Magma's indexing of Hecke characters uses randomness, so in other instantiations, the desired Hecke character may be a different product of the generators DG.$i$.  One must locate the correct Hecke character each time one reopens Magma.}
 \begin{lstlisting}
 K:=CyclotomicField(9);
 I:=Factorization(3*IntegerRing(K))[1][1]^4;
 HG:=HeckeCharacterGroup(I);
 DG:=DirichletGroup(I);
 GR:=Grossencharacter(HG.0,DG.1*DG.2*DG.3,[[3,0],[2,1],[1,2]]);
 \end{lstlisting}
We evaluated the $L$-function of this Grossencharacter at its center using the available tools in Magma (see \cite{magmacode} for our code). The sign in the functional equation turns out to be $-1$ and the order of vanishing at the center turns out to be $1$. 
\end{proof}

\begin{remark}
 Beauville and Schoen \cite{beauvilleschoen} proved that $\kappa(C^9_{1,2})$ is torsion, but the order of vanishing of the $L$-function is 1 and not 0. What null-homologous 1-cycle on $J^{9}_{1,2}$ accounts for the vanishing of $L(M^9_{1,2}/\Q(\zeta_9)^+,2)$, as predicted by Beilinson--Bloch?
\end{remark}

\subsubsection{The $L$-function of $C^{12}_{1,3}$}\label{s:C12}

The computation in this case is somewhat similar, except that the Jacobian is not simple this time. In this subsection $C = C^{12}_{1,3}$ and $J = J^{12}_{1,3}$.  Let $\zeta_3$ be a third root of unity and define $\zeta_{12}=i\zeta_3$.
Fix a primitive character $\chi : \mu_{12} \lra \C^\times$ by sending $\zeta_{12}$ to $e^{\frac{\pi i}{6}}$. Identify as usual $\Gal(\Q(\zeta_{12})/\Q)$ with $(\Z/12\Z)^\times=\{ 1,5,7,11 \}$ and order this set by conjugate pairs as $\{ (1,11), (5,7) \}$. Note that $\Q(\zeta_{12})=\Q(\zeta_3, i)$ is biquadratic. 

Let $V = H^{1,0}(C)$ as before. A basis for $f^* V$ is  
\begin{equation}\label{L12}
L^{12}_{1,3} = \{ w^{12}_{1,3}, w^{12}_{2,6}, w^{12}_{5,3} \}.
\end{equation}
The automorphism $\sigma = \sigma_{1,4}$ of $F^{12}$ descends to an order 12 automorphism of $C$. We have 
\begin{equation}\label{eigenC12}
\sigma^*w^{12}_{1,3}=\zeta_{12} w^{12}_{1,3}, \qquad \sigma^*w^{12}_{2,6}=\zeta_{12}^2 w^{12}_{2,6}, \qquad \sigma^*w^{12}_{5,3}=\zeta_{12}^5 w^{12}_{5,3}.
\end{equation}

For each prime $\ell$, the first $\ell$-adic cohomology $H^1_{\et}(J_{\bar \Q}, \bar\Q_\ell)$ inherits an action of $\mu_{12}\subset \Aut(C)$ and admits an eigenspace decomposition
\[
H^1_{\et}(J_{\bar \Q}, \bar\Q_\ell)= \bigoplus_{\phi} H^{\phi},
\]
where the sum is over (not necessarily primitive) characters $\phi : \mu_{12} \lra \C^\times$ and $H^\phi$ is the corresponding eigenspace.
This decomposition carries an action of $\Gal(\Q(\zeta_{12})/\Q)=(\Z/12\Z)^\times$ by $a(H^\phi)=H^{\phi^a}$. It follows from \eqref{eigenC12} that 
\[
H^{1,0}(C)=H^{\chi}\oplus H^{\chi^2}\oplus H^{\chi^5},
\]
after base change to $\C$.

In general, the differential forms $w^m_{r,s}$ arise from the new part of the Jacobian of $C^m_{r,s}$ if and only if $\gcd(m,r,s)=1$. 
In particular, $w^{12}_{1,3}$ and $w^{12}_{5,3}$ of \eqref{L12} arise from the $2$-dimensional abelian variety $A:=A^{12}_{1,3}$ over $\Q$, while $2\cdot w^{12}_{2,6}$ arises as the pullback of the form $w^6_{1,3}$ on $F^6$ via the map $F^{12}\lra F^6$, $(x,y)\mapsto (x^2, y^2)$. Note that $w^6_{1,3}$ is the regular differential form of the elliptic curve $E:=J^{6}_{1,3}=A^6_{1,3}$ over $\Q$.
We have an isogeny $J\simeq A \times E$ over $\Q$.

The elliptic curve $E$ has CM by the imaginary quadratic field $\Q(\zeta_6)=\Q(\zeta_3)$ with CM type $\{ 1 \}$. The $L$-function of $E$ is given by the $L$-function of a Grossencharacter for $\Q(\zeta_3)$ of infinity type $[1,0]$, namely the Jacobi sum $\tau^{6,\chi}_{1,3}$. The Jacobi sum $\tau^{12,\chi^2}_{1,3}=\tau^{12,\chi}_{2,6}$ is a Grossencharacter for $\Q(\zeta_{12})$ and satisfies 
\[
\tau^{12,\chi}_{2,6}=\tau^{6,\chi}_{1,3}\circ \Norm^{\Q(\zeta_{12})}_{\Q(\zeta_3)} : I_{\Q(\zeta_{12})}(6)\lra \C^\times.
\]
Since $\Gal(\Q(\zeta_{12})/\Q(\zeta_3))=\{ 1,7 \}$, it follows that $\tau^{12,\chi}_{2,6}$ has infinity type $[[1,0],[0,1]]$. The submotive $H^{\chi^2}\oplus H^{\chi^{10}}$ of $H$ is defined over $\Q(i)$ and arises from the elliptic curve $E$.
Denoting by $$\xi : \Gal(\Q(\zeta_{12})/\Q(\zeta_3))=\Gal(\Q(i)/\Q)\lra \C^\times$$ the non-trivial character, we observe that 
\begin{multline*}
L(H^{\chi^2}/\Q(\zeta_{12}),s)=L((H^{\chi^2}\oplus H^{\chi^{10}})/\Q(i),s)=L(E/\Q(i),s) \\ =L(E/\Q, s)L(E/\Q, \xi, s) =L(\tau^{6,\chi}_{1,3}/\Q(\zeta_3),s)L(\tau^{6,\chi}_{1,3}/\Q(\zeta_3),\xi,s)=L(\tau^{12, \chi}_{2,6}/\Q(\zeta_{12}),s).
\end{multline*}

The abelian variety $A=A^{12}_{1,3}$ is isogenous over $\C$ to the product of two isomorphic elliptic curves with CM by $\Q(i)$ and CM type $\{ 1\}$. We see that $H^{1,0}(A)=H^{\chi}\oplus H^{\chi^5}$ is defined over $\Q(i)$ and $H^{12}_{1,3}=\{ 1,5\}=\Gal(\Q(\zeta_{12})/\Q(i))$. We have 
\[
L(A/\Q,s)=L(\tau^{12,\chi}_{1,3}/\Q(\zeta_{12}), s),
\] 
and the infinity type of $\tau^{12,\chi}_{1,3}$ is $[[1,0],[1,0]]$.
The motive
\[ 
M(A):=H^{2,0}(A)\oplus H^{0,2}(A)=(H^{\chi}\otimes H^{\chi^5})\oplus (H^{\chi^{11}}\otimes H^{\chi^7})
\] 
is defined over $\Q$ and $L(M(A)/\Q,s)=L(\tau^{12,\chi}_{1,3}\tau^{12,\chi}_{5,3}/\Q(i),s)$.

Let $H$ denote the motive $H^1(J)$.
The motive of interest is the $(3,0)+(0,3)$ part of $\bigwedge^3 H$,
\[
M=M^{12}_{1,3}=(H^{\chi}\otimes H^{\chi^2}\otimes H^{\chi^5})\oplus (H^{\chi^{11}}\otimes H^{\chi^{10}}\otimes H^{\chi^7}),
\]
which is defined over $\Q(\zeta_{12})^+ = \Q(\sqrt{3})$. 
The Beilinson--Bloch conjecture \eqref{bb2} predicts that
\begin{equation}
\rank_{\Z} \Gr^2(J_{\Q(\sqrt{3})}) = \ord_{s=2} L(M/\Q(\sqrt{3}), s).
\end{equation}

\begin{proposition}\label{prop:L12}
We have 
$
L(M/\Q(\sqrt{3}), 2)=L((M(A)\otimes H^1(E))/\Q, 2) \neq 0.$
\end{proposition}

\begin{proof}
The motive
\begin{multline*}
M(A)\otimes H^1(E)=(H^{\chi}\otimes H^{\chi^2}\otimes H^{\chi^5})\oplus (H^{\chi}\otimes H^{\chi^5}\otimes H^{\chi^{10}}) \\ \oplus (H^{\chi^2}\otimes H^{\chi^{11}}\otimes H^{\chi^7})\oplus (H^{\chi^{10}}\otimes H^{\chi^{11}}\otimes H^{\chi^7})
\end{multline*}
is defined over $\Q$, and the equality of $L$-functions $L(M/\Q(\sqrt{3}), s)=L((M(A)\otimes H^1(E))/\Q, s)$ is clear. To compute the central value, we proceed as in the proof of Proposition \ref{prop:L9}. Define the Grossencharacter $\psi = \tau^{12,\chi}_{1,3}\tau^{12,\chi}_{5,3}\tau^{12,\chi}_{2,6}$. 
Then we have 
\[
L(M/\Q(\sqrt{3}), s) = L(\psi/\Q(\zeta_{12}), s),
\]
and the infinity type of $\psi$ is $[[3,0],[2,1]]$.
By comparing values at primes above 
\[p\in \{ 5,7,11,13,17,19,23,37,73 \},\]
we find in Magma that $\psi$ is a character whose conductor is the unique prime ideal in $\Z[\zeta_{12}]$ of norm 4 (see \cite{magmacode}).  
Magma then reports that the sign of the functional equation of $L(M/\Q(\sqrt{3}), s)$ is $+1$ and \[L(M/\Q(\sqrt{3}), 2) = L(\psi/\Q(\zeta_{12}),2) \approx 0.724.\]
In particular, the $L$-value is non-zero.
\end{proof}
Proposition \ref{prop:L12} and the  Beilinson--Bloch conjecture \eqref{bb2} lead to the prediction that the group $\Gr^2(J_{\Q(\sqrt{3})})$ is torsion, and  in particular that the class $\kappa(C)$ is torsion, which was Theorem \ref{thm:Griffiths}. 
This example therefore gives new evidence for the Beilinson--Bloch conjecture (\ref{bb2}). 

\begin{remark}
We saw above that $\psi$ has conductor of norm 4, which is remarkably small. Thus, this curve $C^{12}_{1,3}$ shows that 
\begin{enumerate}[$(a)$]
    \item the group $\Gr_1(J)$   can (and assuming the Beilinson--Bloch conjecture, does) have rank 0 sometimes, even for non-hyperelliptic Jacobians, and
    \item we need only look at one of the first non-hyperelliptic curves to find such rank 0 examples!
\end{enumerate} 
To make $(b)$ more precise, we first recall a phenomenon which can be observed in tables of elliptic curves: if $N_r$ is the minimal conductor of an elliptic curve $E/\Q$ with Mordell--Weil rank $r$, then $N_r < N_{r+1}$.  One can (conjecturally) define analogous integers $N_r$ for the ranks of the Griffiths group $\Gr_1(J)$ of Jacobians $J$ of genus $g \geq 3$ curves $C/\Q$, using the Galois representation $F^0/F^1$ from  Section \ref{sec:BB}  to define a notion of ``Griffiths group conductor''.
Then $D^{12}$ is one of the first non-hyperelliptic curves in the sense that it has small Griffiths group conductor. Does the curve $D^{12}$ have minimal Griffiths group conductor among all non-hyperelliptic genus 3 curves over $\Q$? 
\end{remark}

\subsubsection{The $L$-function of $C^{15}_{1,5}$}
The curve $C^{15}_{1,5}$ has genus  4.  Since  $H^{4,1}(J)$ is 4-dimensional, the relevant $L$-function $L(M^{15}/\Q(\zeta_{15})^+,s)$ is a product of four $L$-functions attached to Hecke characters for the degree 8 field $\Q(\zeta_{15})$, with norm-conductors $25,81,2025,$ and $2025$ (see \cite{magmacode}). With the help of Drew Sutherland, we were able to compute these $L$-functions and their special values. Two of them were non-vanishing, but the two of norm-conductor 2025 had sign $-1$ and order of vanishing $1$. Thus, the order of vanishing of $L(M^{15}/\Q(\zeta_{15})^+,s)$ is $2$, and Beilinson--Bloch predicts the existence of two linearly independent elements of infinite order in the Griffiths group $\Gr_1(J)$, despite the fact that the Ceresa class is predicted to be torsion. We say ``predicted'', because in this case we do not have a proof that $\kappa(C)$ is torsion in $\Gr_1(J)$, only that the Abel--Jacobi image is torsion. 

\section{Hyperelliptic Isogenies}\label{s:speculation}
It is of course desirable to prove that the Ceresa cycles $\kappa_{(0,0)}$ for $D^9$, $D^{12}$, and $D^{15}$ are torsion already in the Chow group (not just their images in the Griffiths group and intermediate Jacobian). We would also like to understand, more generally, why certain non-hyperelliptic curves have torsion Ceresa cycles, and whether there is a way to characterize them. During our computations, we noticed that these curves share a certain property that seems relevant for these questions.  Namely, each of these curves admits a non-zero algebraic correspondence to the hyperelliptic curve $C^m_{1,1} \colon y^m= x(1-x)$.  In fact, this property nearly characterizes the curves $D^9$, $D^{12}$, and $D^{15}$, among all non-hyperelliptic cyclic Fermat quotients, as we will explain.

Let $J^m_{a,b}$ denote the Jacobian of $C^m_{a,b}$ and recall the new part $A^m_{a,b}$, which has CM by $\Z[\zeta_m]$.  We say two cyclic quotients $C^m_{a,b}$ and $C^m_{a',b'}$ are {\it isogenous} if their corresponding new parts are isogenous over $\bar\Q$.\footnote{A cycle in $\CH_1(J)$ is torsion if and only if its image in $\CH_1(J_{ \bar \Q})$ is torsion.  So  it is enough, for our purposes in this section, to work over $\bar \Q$.}  The isogeny relation among the curves $C^m_{a,b}$ is a coarser relation than the equivalence relation $\sim_m$ introduced in Section \ref{s:fermat}. To check whether two such curves are isogenous, it is enough to check that the corresponding new parts have the same CM types, up to the action of $\Gal(\Q(\zeta_m)/\Q)$. Recall from Section \ref{s:jacobisum} that the CM type of $A^m_{a,b}$ can be identified with the set
\[
H^{m}_{a,b}:=\{ h\in (\Z/m\Z)^\times \colon \langle ha \rangle+\langle hb \rangle < m \}.
\]
The group $\Gal(\Q(\zeta_m)/\Q) \simeq (\Z/m\Z)^\times$ acts on the set of CM types by scaling.  

\begin{lemma}\label{l:hyperellipticisogeny}
Up to isogeny, the unique hyperelliptic cyclic quotient of $F^m$ is $C^m_{1,1}$.
\end{lemma}
\begin{proof}
For $m$ odd this follows from Proposition \ref{p:hyperelliptic}. For $m$ even, we observe that $A^m_{1,1}$ is isogenous to $A^m_{1,m/2}$ since they both have CM type $\{ t \in (\Z/m\Z)^\times \colon \langle t \rangle < m/2\}$, by direct computation.  Thus, this case also follows from Proposition \ref{p:hyperelliptic}. 
\end{proof}

We can classify those non-hyperelliptic curves $C^m_{a,b}$ which are isogenous to a hyperelliptic one. A result of Koblitz and Rohrlich implies that this can only happen if $\gcd(m,6) > 1$ \cite{koblitzrohrlich}.  In fact, there are only finitely many examples with $m\not \equiv 2\pmod{4}$:

\begin{proposition}\label{prop:hypisog}
Suppose $C = C^m_{a,b}$ is non-hyperelliptic and $m \not\equiv2\pmod{4}$. Then $C$ is isogenous to $C^m_{1,1}$ if and only if $(m,a,b)$ is equivalent to one of the following:
\[\{(9,1,2),(12,1,3), (15,1,5), (21,1,2), (21,1,3), (24,1,5), (24,1,7), (60,1,10), (60,1,19)\}.\]
\end{proposition}

\begin{proof}
That there are no examples with $m > 180$ is a special case of a result of Aoki \cite[Thm.\ 0.1]{aoki}, who determined when two cyclic Fermat quotients are isogenous to each other. From staring at his classification, we quickly rule out any ``exceptional isogenies'' from $C^m_{1,1}$, with $m \not\equiv 2\pmod{4}$, to a non-hyperelliptic $C^m_{a,b}$. 
Indeed, $C^{m}_{a,b}$ being non-hyperelliptic implies that $(a,b)\nsim_m (1,1)$ and $(a,b)\nsim_m (1,m/2)$ in case $m$ is even (by Proposition \ref{p:hyperelliptic}). Moreover, for $m$ even, the condition $m \not\equiv2\pmod{4}$ forces $\gcd(m, m/2-1)=1$, so that $(1,1)\sim_m (m/2-1,m/2-1)$. This rules out the first three possibilities for the equivalence class of $(a,b)$ in \cite[Thm.\ 0.1 (3)]{aoki}. The fourth and fifth possibilities are ruled out by the fact that $1$ is odd and $m \not\equiv2\pmod{4}$ respectively.

For $m \leq 180$, it is enough to check which of the finitely many equivalence classes of curves $C^m_{a,b}$ has new part $A^m_{a,b}$ with the same CM type as $A^m_{1,1}$, up to the action of $\Gal(\Q(\zeta_m)/\Q)\simeq (\Z/m\Z)^\times$.  It turns out that up to equivalence the only non-hyperelliptic curves $C^m_{a,b}$ with this CM type are those listed (see \cite{sagecode} for the relevant code). 
\end{proof}

Notice that the first three curves on this list are $D^9$, $D^{12}$, and $D^{15}$.  When $m \equiv 2\pmod{4}$, the situation is somewhat different, since there are actually infinitely many examples.

\begin{proposition}
Suppose $C = C^m_{a,b}$ is non-hyperelliptic and $m = 4k+2$, for some $k \geq 0$. Then $C$ is isogenous to $C^m_{1,1}$ if and only if $(m,a,b)$ is equivalent to $(m,1,k)$ or one of the following:
\[(14,1,2),(18,1,2),(18,1,5),(30,1,2),(30,1,3), (30,1, 4), (30,1,8), (30,2, 3),\]
\[
(42,1, 2), (42,1, 4) ,(42,1, 5), (42,1, 8), (42,1, 11), (42,1,15), (78,1,16). \]
\end{proposition}
\begin{proof}
The proof is similar. The infinite family $C^m_{1,k}$ can again be read off of Aoki's result.
\end{proof}

We say a non-hyperelliptic $C^m_{a,b}$ is {\it minimal} if it does not cover any non-hyperelliptic cyclic Fermat quotient of lower genus. For our purposes, the minimal non-hyperelliptic $C^m_{a,b}$ are the first curves to study, since if the Ceresa cycle of a curve has infinite order, so will the Ceresa cycle of any cover. For example, the curve $C^{21}_{1,2}$ listed above is isogenous to a hyperelliptic curve, but it also covers the Klein quartic curve $C^7_{1,2}$. Since the latter has infinite order Ceresa cycle \cite{kimura, tadokoro}, so does $C^{21}_{1,2}.$

\begin{proposition}
If $C = C^m_{a,b}$ is a minimal non-hyperelliptic cyclic Fermat quotient which is isogenous to $C^m_{1,1}$, then $C$ is equivalent to one of the following curves:
\[\{(9,1,2),(12,1,3), (15,1,5), (21,1,3)\} \cup \{ (4k+2,1,k) \colon k \in \Z\}.\]
\end{proposition}
\begin{proof}
For each curve $(m,a,b)$ in Proposition \ref{prop:hypisog}, we must determine whether there exists a divisor $d \mid m$, such that $C^m_{da,db}$ is non-hyperelliptic. We have $C^m_{da,db} \simeq C^{m/h}_{a',b'}$ where $h = \gcd(\langle da\rangle,\langle db\rangle,\langle d(m-a-b)\rangle)$, $a' = \langle da\rangle/h$ and $b' = \langle db \rangle/h$.  For example, the only interesting subcover of $C^{21}_{1,3}$ is $C^7_{1,3} \simeq C^7_{1,1}$, which is hyperelliptic, so it is minimal.  We must do the same to the 15 sporadic examples with $m \equiv 2\pmod{4}$, but we find that they are all non-minimal.
\end{proof}
Thus, the curves $D^9$, $D^{12}$,$D^{15}$, and $C^{21}_{1,3}$ are characterized as the minimal non-hyperelliptic cyclic Fermat quotients $C^m_{a,b}$ with $m \not\equiv 2\pmod{4}$ that are isogenous to a hyperelliptic one.  
\begin{remark}
Not  all curves of the form $C^{4k+2}_{1,k}$ are minimal, but infinitely many of them are. For example, if $4k+2 = 2p$ for some prime $p$, then the curve is minimal.
 \end{remark}

\subsection{Future directions}
It is natural to wonder whether there is an algebraic correspondence from $C^m_{1,1}$ to $D^m$ which sends the Ceresa class of $C^m_{1,1}$ (in the Chow group of $J^m_{1,1}$) to that of $D^m$ or to some multiple of it.  This would prove that the  Ceresa cycle is torsion in these cases, and it would give a nice geometric explanation as well.  In other words, even though the curves $D^m$ are not hyperelliptic, perhaps the motives $H^{2g-3}(\Jac(D^m))$ are hyperelliptic in a certain sense.

Of course, the presence of an isogeny from a non-hyperelliptic Jacobian to a hyperelliptic Jacobian is not by itself noteworthy. However, in our examples the isogeny intertwines an order $m$ automorphism on both curves. In particular, the algebraic correspondence which realizes the isogeny (viewed as a higher genus curve covering both $C^m_{1,1}$ and $D^m$) should itself carry an order $m$ automorphism and should itself be a cyclic Fermat quotient of level $km$ for some $k \geq 1$. Given that the presence of a high order automorphism is what allows us to detect the torsion Ceresa cycle in the first place, this seems like a fruitful direction to explore further.  The next step would be to find the explicit cyclic Fermat quotient models for these algebraic correspondences, so that one can compute pullbacks of Ceresa cycles.    

We hope to study these questions in future work. It would also be interesting to compute numerically the Abel--Jacobi image of the Ceresa cycle for the genus 9 curve $C^{21}_{1,3}$, to see whether it too is torsion, despite the fact that Beauville's method does not apply. 

\subsection*{Acknowledgements} 
We thank A.\ Beauville for discussing the proof of \cite{beauvilleschoen}, and we thank A.\  Sutherland for kindly helping us with the more intensive $L$-function computations. During the preparation of this work, the first author was supported by an Emily Erskine Endowment Fund Postdoctoral Research Fellowship. The second author
was supported by the Israel Science Foundation (grant No. 2301/20).

\phantomsection
\bibliographystyle{plain}
\bibliography{ExpCerRevised.bib}

\begin{thebibliography}{10}

\bibitem{aoki}
N.~Aoki.
\newblock Simple factors of the {J}acobian of a {F}ermat curve and the {P}icard
  number of a product of {F}ermat curves.
\newblock {\em Amer. J. Math.}, 113(5):779--833, 1991.

\bibitem{beauville}
A.~Beauville.
\newblock A non-hyperelliptic curve with torsion {C}eresa class.
\newblock {\em C. R. Math. Acad. Sci. Paris}, 359:871--872, 2021.

\bibitem{beauvilleschoen}
A.~Beauville and C.~Schoen.
\newblock A non-hyperelliptic curve with torsion {C}eresa cycle modulo
  algebraic equivalence.
\newblock {\em Int. Math. Res. Not. IMRN}, 2021.
\newblock to appear.

\bibitem{langebirkenhake}
C.~Birkenhake and H.~Lange.
\newblock {\em Complex abelian varieties}, volume 302 of {\em Grundlehren der
  mathematischen Wissenschaften [Fundamental Principles of Mathematical
  Sciences]}.
\newblock Springer-Verlag, Berlin, second edition, 2004.

\bibitem{blls}
D.~Bisogno, W.~Li, D.~Litt, and P.~Srinivasan.
\newblock Group-theoretic {J}ohnson classes and a non-hyperelliptic curve with
  torsion {C}eresa class.
\newblock Preprint, arXiv:2004.06146, 2020.

\bibitem{bloch}
S.~Bloch.
\newblock Algebraic cycles and values of {$L$}-functions.
\newblock {\em J. Reine Angew. Math.}, 350:94--108, 1984.

\bibitem{bloch85}
S.~Bloch.
\newblock Algebraic cycles and values of {$L$}-functions. {II}.
\newblock {\em Duke Math. J.}, 52(2):379--397, 1985.

\bibitem{Magma}
W.~Bosma, J.~Cannon, and C.~Playoust.
\newblock The {M}agma algebra system. {I}. {T}he user language.
\newblock {\em J. Symbolic Comput.}, 24(3-4):235--265, 1997.
\newblock Computational algebra and number theory (London, 1993).

\bibitem{cainthesis}
C.~Cain.
\newblock {\em {$K$}-{T}heory of {F}ermat curves}.
\newblock 2016.
\newblock Thesis (Ph.D.)--University of Cambridge, Churchill College.

\bibitem{ceresa}
G.~Ceresa.
\newblock {$C$} is not algebraically equivalent to {$C^{-}$} in its {J}acobian.
\newblock {\em Ann. of Math. (2)}, 117(2):285--291, 1983.

\bibitem{colemanfermat}
R.~F. Coleman.
\newblock Torsion points on abelian \'{e}tale coverings of {${\bf
  P}^1-\{0,1,\infty\}$}.
\newblock {\em Trans. Amer. Math. Soc.}, 311(1):185--208, 1989.

\bibitem{Debarre}
O.~Debarre.
\newblock {\em Tores et vari\'{e}t\'{e}s ab\'{e}liennes complexes}, volume~6 of
  {\em Cours Sp\'{e}cialis\'{e}s [Specialized Courses]}.
\newblock Soci\'{e}t\'{e} Math\'{e}matique de France, Paris; EDP Sciences, Les
  Ulis, 1999.

\bibitem{eskandarimurty2}
P.~Eskandari and V.~K. Murty.
\newblock On {C}eresa cycles of {F}ermat curves.
\newblock {\em J. Ramanujan Math. Soc.}, 36(4):363--382, 2021.

\bibitem{eskandarimurty}
P.~Eskandari and V.~K. Murty.
\newblock On the harmonic volume of {F}ermat curves.
\newblock {\em Proc. Amer. Math. Soc.}, 149(5):1919--1928, 2021.

\bibitem{FLV}
L.~Fu, R.~Laterveer, and C.~Vial.
\newblock Multiplicative {C}how-{K}\"{u}nneth decompositions and varieties of
  cohomological {K}3 type.
\newblock {\em Ann. Mat. Pura Appl. (4)}, 200(5):2085--2126, 2021.

\bibitem{Fujiki}
A.~Fujiki.
\newblock On resolutions of cyclic quotient singularities.
\newblock {\em Publ. Res. Inst. Math. Sci.}, 10(1):293--328, 1974/75.

\bibitem{fulton}
W.~Fulton.
\newblock {\em Intersection theory}, volume~2 of {\em Ergebnisse der Mathematik
  und ihrer Grenzgebiete (3) [Results in Mathematics and Related Areas (3)]}.
\newblock Springer-Verlag, Berlin, 1984.

\bibitem{grossFM}
B.~H. Gross.
\newblock {T}he {F}ricke-{M}acbeath curve and triple product {$L$}-functions.
\newblock Slides of a talk,
  \url{http://www.fields.utoronto.ca/talks/Fricke-Macbeath-curve-and-triple-product-L-functions},
  2021.

\bibitem{kimura}
K.~Kimura.
\newblock On modified diagonal cycles in the triple products of {F}ermat
  quotients.
\newblock {\em Math. Z.}, 235(4):727--746, 2000.

\bibitem{koblitzrohrlich}
N.~Koblitz and D.~Rohrlich.
\newblock Simple factors in the {J}acobian of a {F}ermat curve.
\newblock {\em Canadian J. Math.}, 30(6):1183--1205, 1978.

\bibitem{magmacode}
D.~T.-B.~G. Lilienfeldt and A.~Shnidman.
\newblock {\verb^Magma^ code related to this paper}.
\newblock \url{https://math.huji.ac.il/~lilienfeldt/CeresaMagmaCode.rtf}.

\bibitem{sagecode}
D.~T.-B.~G. Lilienfeldt and A.~Shnidman.
\newblock {\verb^Sage^ code related to this paper}.
\newblock \url{http://math.huji.ac.il/~shnidman/ceresacode.py}.

\bibitem{lorenzogarcia}
E.~Lorenzo~Garc\'{\i}a.
\newblock Twists of non-hyperelliptic curves of genus 3.
\newblock {\em Int. J. Number Theory}, 14(6):1785--1812, 2018.

\bibitem{MSSV}
K.~Magaard, T.~Shaska, S.~Shpectorov, and H.~V\"{o}lklein.
\newblock The locus of curves with prescribed automorphism group.
\newblock Number 1267, pages 112--141. 2002.
\newblock Communications in arithmetic fundamental groups (Kyoto, 1999/2001).

\bibitem{mumfordprym}
D.~Mumford.
\newblock Prym varieties. {I}.
\newblock In {\em Contributions to analysis (a collection of papers dedicated
  to {L}ipman {B}ers)}, pages 325--350. 1974.

\bibitem{otsubo}
N.~Otsubo.
\newblock On the {A}bel-{J}acobi maps of {F}ermat {J}acobians.
\newblock {\em Math. Z.}, 270(1-2):423--444, 2012.

\bibitem{qiuzhang}
C.~Qiu and W.~Zhang.
\newblock {I}njectivity of the {A}bel-{J}acobi map and {G}ross-{K}udla-{S}choen
  cycles.
\newblock In preparation.

\bibitem{Reid}
M.~Reid.
\newblock Canonical {$3$}-folds.
\newblock In {\em Journ\'{e}es de {G}\'{e}ometrie {A}lg\'{e}brique d'{A}ngers,
  {J}uillet 1979/{A}lgebraic {G}eometry, {A}ngers, 1979}, pages 273--310.
  Sijthoff \& Noordhoff, Alphen aan den Rijn---Germantown, Md., 1980.

\bibitem{tadokoro}
Y.~Tadokoro.
\newblock A nontrivial algebraic cycle in the {J}acobian variety of the {K}lein
  quartic.
\newblock {\em Math. Z.}, 260(2):265--275, 2008.

\bibitem{tadokoro2}
Y.~Tadokoro.
\newblock Nontrivial algebraic cycles in the {J}acobian varieties of some
  quotients of {F}ermat curves.
\newblock {\em Internat. J. Math.}, 27(3):1650027, 10, 2016.

\bibitem{TateThesis}
J.~T. Tate, Jr.
\newblock {\em F{OURIER} {ANALYSIS} {IN} {NUMBER} {FIELDS} {AND} {HECKE}'{S}
  {ZETA}-{FUNCTIONS}}.
\newblock ProQuest LLC, Ann Arbor, MI, 1950.
\newblock Thesis (Ph.D.)--Princeton University.

\bibitem{sagemath}
{The Sage Developers}.
\newblock {\em {S}ageMath, the {S}age {M}athematics {S}oftware {S}ystem
  ({V}ersion 9.2)}, 2020.
\newblock {\tt https://www.sagemath.org}.

\bibitem{weiljacobi}
A.~Weil.
\newblock Jacobi sums as ``{G}r\"{o}ssencharaktere''.
\newblock {\em Trans. Amer. Math. Soc.}, 73:487--495, 1952.

\end{thebibliography}

\end{document}